%% file: article_plain.tex
\documentclass[11pt]{article}
\pdfoutput=1 
\usepackage{amssymb}
\usepackage[margin=1in]{geometry}

\usepackage{lineno}
\usepackage{graphicx}

\usepackage{moreverb,url}
\usepackage{makecell}
\usepackage[breaklinks=true, colorlinks=true, linkcolor={blue!90!black}, citecolor={blue!90!black}, urlcolor={blue!90!black}]{hyperref}

\newcommand\BibTeX{{\rmfamily B\kern-.05em \textsc{i\kern-.025em b}\kern-.08em
T\kern-.1667em\lower.7ex\hbox{E}\kern-.125emX}}

\newcommand{\tn}{\textnormal}

\usepackage[algo2e,linesnumbered,vlined,algoruled]{algorithm2e}

\usepackage{placeins}
\SetCommentSty{mycommfont}

\usepackage{geometry,xcolor,tabularx,nccmath}
\usepackage{amsmath,amsthm,amssymb,enumitem}
\usepackage{booktabs,caption,subcaption,mathtools,multirow}
\usepackage{graphicx,tabulary}

\usepackage[utf8]{inputenc}
\usepackage{amsfonts}
\usepackage{fullpage}

\newcommand{\minimize}{\operatornamewithlimits{minimize}}
\newcommand{\maximize}{\operatornamewithlimits{maximize}}

\newcommand{\cC}{\mathcal{C}}

\newcommand{\cE}{\mathcal{E}}

\newcommand{\cM}{\mathcal{M}}

\newcommand{\cP}{\mathcal{P}}

\newcommand{\cR}{\mathcal{R}}

\newcommand{\cV}{\mathcal{V}}
\newcommand{\f}{f_{u^*,v^*,i^*,j^*}}

\newcommand{\bx}{\mathbf{x}}
\newcommand{\Tbatch}{B}

\usepackage[colorinlistoftodos]{todonotes}

\usepackage{placeins}
\usepackage{etoolbox}

\newtheorem{lemma}{Lemma}

\usepackage[nameinlink]{cleveref}
\hypersetup{colorlinks={true},allcolors={blue}}

\crefformat{section}{\S#2#1#3}
\crefformat{subsection}{\S#2#1#3}
\crefformat{equation}{#2Equation~\textup{(#1)}#3}
\crefformat{cons}{#2\textbf{Constraint}~\textup{\textbf{(#1)}}#3}
\labelcrefformat{cons}{\textup{#2(#1)#3}}
\crefformat{consset}{#2\textbf{Constraints}~\textup{\textbf{(#1)}}#3} 
\labelcrefformat{consset}{\textup{\textbf{#2(#1)#3}}}
\crefformat{func}{#2Function~\textup{(#1)}#3}
\labelcrefformat{func}{#2Function~\textup{(#1)}#3}
\crefformat{algo}{#2\textbf{Algorithm}~\textup{\textbf{#1}}#3}
\crefformat{objfunc}{#2\textbf{Objective function}~\textup{\textbf{(#1)}}#3}
\labelcrefformat{objfunc}{\textup{(#2#1#3)}}
\crefformat{ch}{#2Chapter~\textup{#1}#3}
\crefformat{sec}{#2Section~\textup{#1}#3}
\crefformat{tab}{#2Table~\textup{#1}#3}
\crefformat{fig}{#2Figure~\textup{#1}#3}
\crefformat{eq}{#2equation~\textup{(#1)}#3}
\labelcrefformat{eq}{\textup{(#2#1#3)}}
\crefformat{theo}{#2Theorem~\textup{#1}#3}
\crefformat{lem}{#2Lemma~\textup{#1}#3}
\crefformat{app}{#2Appendix~\textup{#1}#3}
\crefformat{scenario}{#2Scenario~\textup{#1}#3}
 \AtBeginDocument{%
    \crefname{figure}{Figure}{figures}%
}

\expandafter\def\expandafter\UrlBreaks\expandafter{\UrlBreaks
  \do\a\do\b\do\c\do\d\do\e\do\f\do\g\do\h\do\i\do\j%
  \do\k\do\l\do\m\do\n\do\o\do\p\do\q\do\r\do\s\do\t%
  \do\u\do\v\do\w\do\x\do\y\do\z\do\A\do\B\do\C\do\D%
  \do\E\do\F\do\G\do\H\do\I\do\J\do\K\do\L\do\M\do\N%
  \do\O\do\P\do\Q\do\R\do\S\do\T\do\U\do\V\do\W\do\X%
  \do\Y\do\Z}

\DeclareMathOperator*{\argmin}{arg\,min}

\usepackage{amsmath,cases,eqparbox,xparse}

\makeatletter
\NewDocumentCommand{\eqmathbox}{o O{c} m}{%
  \IfValueTF{#1}
    {\def\eqmathbox@##1##2{\eqmakebox[#1][#2]{$##1##2$}}}
    {\def\eqmathbox@##1##2{\eqmakebox{$##1##2$}}}
  \mathpalette\eqmathbox@{#3}
}
\makeatother

\makeatletter
\let\@msm@th@eqref\eqref
\renewcommand{\eqref}[1]{%
  \begingroup
  \leavevmode
  \color{blue}%
  \hypersetup{linkbordercolor=[named]{blue}}%
  \@msm@th@eqref{#1}%
  \endgroup
}
\makeatother

\usepackage[normalem]{ulem} 
 
\crefname{defi}{definition}{definitions}
\Crefname{defi}{Definition}{Definitions}
\crefname{lemma}{lemma}{lemmas}
\Crefname{lemma}{Lemma}{Lemmas}
\crefname{assumption}{assumption}{assumptions}
\Crefname{assumption}{Assumption}{Assumptions}

\usepackage{multicol}

\usepackage{authblk}
\providecommand{\keywords}[1]{\noindent\textbf{Keywords:} #1}

\usepackage{natbib}

\begin{document}

\title{Large-Scale Dynamic Ridesharing with Iterative Assignment}


\author[1,2]{Akhil Vakayil}

\affil[1]{Argonne National Laboratory, 9700 S. Cass Ave., Lemont, IL 60439, USA }
\author[1]{Felipe de Souza}
\author[1,3]{Taner Cokyasar}
\author[1]{\\Krishna Murthy Gurumurthy}
\author[1]{Jeffrey Larson}

\affil[2]{Georgia Institute of Technology, North Ave. NW, Atlanta, GA 30332, USA }
\affil[3]{TrOpt R\&D, Balcali mah., Saricam, 01330, Adana, Turkey}

\maketitle


\begin{abstract}
Transportation network companies (TNCs) have become a highly utilized transportation mode over the past years. At their emergence, TNCs were serving ride requests one by one. However, the economic and environmental benefits of ridesharing encourages them to dynamically pool multiple ride requests to enable people to share vehicles. In a dynamic ridesharing (DRS) system, a fleet operator seeks to minimize the overall travel cost while a rider desires to experience a faster (and cheaper) service. While the DRS may provide relatively cheaper trips by pooling requests, the service speed is contingent on the objective of the vehicle-to-rider assignments. Moreover, the operator must quickly assign a vehicle to requests to prevent customer loss. In this study we develop an iterative assignment (IA) algorithm with a balanced objective to conduct assignments quickly. A greedy algorithm from the literature is also tailored 
to further reduce the computational time. The IA was used to measure the impact on service quality of fleet size; assignment frequency;  the weight control parameter of the two objectives on vehicle occupancy---rider wait time and vehicle hours traveled. A case study in Austin, TX, reveals that the key performance metrics are the most sensitive to the weight parameter in the objective function.
\end{abstract}


\input{content}

\clearpage
\bibliographystyle{elsarticle-harv} 
\bibliography{refs}

\vfill
\framebox{\parbox{.90\linewidth}{\scriptsize The submitted manuscript has been created by
        UChicago Argonne, LLC, Operator of Argonne National Laboratory (``Argonne'').
        Argonne, a U.S.\ Department of Energy Office of Science laboratory, is operated
        under Contract No.\ DE-AC02-06CH11357.  The U.S.\ Government retains for itself,
        and others acting on its behalf, a paid-up nonexclusive, irrevocable worldwide
        license in said article to reproduce, prepare derivative works, distribute
        copies to the public, and perform publicly and display publicly, by or on
        behalf of the Government.  The Department of Energy will provide public access
        to these results of federally sponsored research in accordance with the DOE
        Public Access Plan \url{http://energy.gov/downloads/doe-public-access-plan}.}}
\end{document}

%% file: content.tex
\section{Introduction}\label[sec]{introduction}
Ridesharing commonly refers to two or more simultaneous trips being carried out in a
single vehicle. This increase in vehicle occupancy is often encouraged to reduce the
number of vehicles on roads~\citep{morency2007ambivalence}. Apart from reducing
traffic congestion, ridesharing can also decrease costs for both customers and vehicle
operators. Hence, transportation network companies (TNCs) offer
pooling options, such as Uber Pool and Lyft Line~\citep{ryan}. A recent
survey reveals that a large majority of riders prefer these pooled
services to their single-ride counterparts because they are ``cheaper" and
``better for the environment"~\citep{sarriera2017share}. To maximize these
benefits, it is essential to have a system-level vehicle-to-rider matching approach that optimizes 
certain objectives, such as minimizing total travel time or delay time.

In a dynamic ridesharing (DRS) environment, customers make ride requests through their smartphones, 
and the TNC operator must \emph{quickly} assign these requests
to vehicles that are either idle or currently serving other requests but with available seats~\citep{AGATZ20111450}. These
assignments must respect a predefined maximum time within which a rider has
to be picked up, as well as the maximum allowed rerouting delay for each rider
currently being served.
Pick-up and drop-off concerns constrain  the TNC operator's possible assignment combinations.
(If a request is not promptly served, customers may
reconsider an alternative transportation mode.)
In large-scale transportation
networks, assigning requests to vehicles quickly while accounting for operator
and rider benefits is burdensome; the main motivation in this study is to
address such limitations. Although making assignments in real time may be a simple approach,
a batched assignment approach in which rider-to-vehicle
assignments are carried out at predetermined time intervals can be
potentially more beneficial to both TNC operators and riders. The reason is that batching may enable
greater coordination of pick-ups and drop-offs.

In this study we propose an iterative assignment \texttt{IA} algorithm to 
assign rider requests to TNC vehicles at discrete time intervals. In each
batching interval, we solve linear programs (LPs) iteratively to minimize two
objectives: total TNC fleet travel time and total customer delay time.
The first objective approaches the problem from a fleet operator standpoint;
the second objective considers a customer-centric assignment.
This approach eliminates the need to solve integer programs (IPs), which are often more time-consuming to solve compared with their linear programming relaxations.
Apart from the \texttt{IA} algorithm, we tailor a greedy
real-time algorithm that assigns requests to vehicles as soon as they emerge in
the communication network \citep{Gurumurthy2020}. The greedy algorithm has a considerably shorter 
time-to-solution and serves as a benchmark to the \texttt{IA}. 
Our numerical tests focus on key parameters, including the time required to solve each problem, 
TNC fleet size, the number of vehicles considered for a given request, the trade-off in the objective between customer-centric and operator-oriented approaches, and the batching interval duration. We provide detailed 
analysis to highlight the impact of these parameters on essential solution metrics, namely, vehicle occupancy and 
response time.


Both the \texttt{IA} and the greedy approaches have been implemented as components of POLARIS, the Planning and Operations
Language for Agent-based Regional Integrated
Simulation~\citep{auld2016polaris} software. POLARIS is an agent-based simulation tool
developed by Argonne National Laboratory's Vehicle and Mobility Systems Department
and is used to quantify the impact of emerging and existing vehicle and
transportation technologies on a variety of metrics, such as vehicle miles
traveled, energy consumed, and greenhouse gas emitted in large metropolitan
areas. Time-to-solution becomes a vital
determinant of the methodology derivation process because the study is devoted
to large-scale DRS, and the solution time of this approach should not pose a
bottleneck in the POLARIS tool.

The outline of the manuscript is as follows. \cref{litrev} provides a thorough
literature review while comparing the solution approaches. \cref{method}
formally describes the problem, assumptions, notations, and solution
approaches. \cref{experiments} contains numerical experiments that measure the
impact of batching interval and the trade-off between the operator-centric and
customer-centric perspectives. We also study the effect of fleet size on key solution metrics
and the impact of number of vehicles considered for each request in 
large-scale computational experiments. \cref{conclusion} summarizes the
work and displays the future of the problem from our perspective.


\section{Literature Review}\label[sec]{litrev}

DRS is a well-studied subject in the transportation literature, and
several review articles have summarized existing
works~\citep{furuhata2013ridesharing,martins2021optimizing,tafreshian2020frontiers,tahmasseby2014dynamic}.
DRS is expected to reduce traffic congestion by increasing vehicle occupancy. \citet{alisoltani2021can} used
a simulation to assess the validity of this perception and found that the DRS can be  impactful in 
large-scale transportation networks, whereas its contribution is slight in small- and medium-scale networks. 
\citet{alonso2017demand} demonstrated a multistep approach to tackle the DRS problem where they first construct a pairwise shareability graph, apply pruning to create a candidate trip's graph, and optimize over the graph using an IP.
Their solution methodology is shown to efficiently match over 800,000 riders with 3,000 vehicles, and they report an average of 2,000 riders in each of their IPs solved.
A major difference in our study is that the \texttt{IA} is capable of assigning more than one rider to a vehicle at each iteration, whereas their model assigns only one rider to each vehicle.

\citet{agatz2012optimization} listed three common objectives in ridesharing
optimization models: (i) minimize total vehicle miles traveled, (ii)
minimize total vehicle times traveled, and (iii) maximize the number of
vehicle-to-rider matches. In this study we minimize the cost incurred by the service provider and customers for serving all incoming requests. The service provider costs account for the marginal travel cost increase (due to detour) for serving a given request. The experienced customer cost is the excess time traveled to reach its destination compared with a direct trip from its origin to its destination. 
\citet{furuhata2013ridesharing} divided
ridesharing patterns into four classes based on spatiotemporal information of riders:
 identical,  inclusive,  partial, and  detour. Here, identical
refers to the same origin and destination (OD) pairs as well as time schedules for
two riders; inclusive means one rider's OD pair is spanned by another rider's
OD pair while time schedules allow such sequential pick-up/drop-off ordering;
partial refers to the case in which one rider's either origin or destination is
different from another but time schedules allow shared service; detour
represents the scenario in which OD pairs for riders are distinct but time
schedules allow shared service. The solution methodologies proposed in this
study consider all of these patterns.

\textit{Path flexibility} allows amending the current vehicle route to pick up a rider. \textit{Multi-hop} refers to transferring a rider from one vehicle to another to gain improvement in the objective. \citet{masoud2017real} categorized the ridesharing studies in the literature based on the flexibility of paths, multi-hop routes, the ability to handle multiple riders, the solution algorithm, and the optimality of the solution. We account for path flexibility and both individual and grouped riders, for example, a group of friends with common OD pairs. We do not account for the multi-hop and find it  an approach that may potentially make riders uncomfortable. We propose both an optimization model finding an optimal solution in a given batching interval and a greedy heuristic to solve the problem. However, these methodologies provide proxies since large-scale problems are considered, and finding the globally optimal solution for an entire day in such problems is difficult if not impossible.

\citet{masoud2017real} stated that many studies considered the rolling horizon
approach in which the TNC operator solves a static ridesharing problem at
discrete time intervals as requests are batched. \citet{agatz2012optimization} mentioned that solving these problems can
become
challenging in large-scale transportation networks as the computational
complexity quickly increases. To address this situation, researchers have used
decomposition, partitioning, and clustering
methods~\citep{agatz2012optimization,masoud2017decomposition,najmi2017novel,nourinejad2016agent,pelzer2015partition,shen2016dynamic}. 
It is known that such partitioning methods may eliminate potential better
assignments, and minimizing the optimality loss requires solving an NP-hard min-cut
problem~\citep{aissi2008complexity,johnson1993min}. In this study we limit
the number of vehicles queried for each request to control the problem complexity. We note that
such an approach may also end up missing better vehicle-to-rider assignments.
However, this limit can be easily tuned to balance the trade-off between the
computational time and the solution quality. 

\section{Methodology}\label[sec]{method}
In this
section we provide notation, state the problem formally, present an assignment
problem with its relaxation, propose the  \texttt{IA} algorithm, address
the possible infeasibility of the LP, shed light on the operator- and customer-centric
objectives, and propose a greedy real-time assignment approach.

\subsection{Notation and formal problem statement}
\label[sec]{notation}

Given a set of available vehicles
$\cV$, let $\cR$ be the set of customer requests that arrive between the start
and end of our planning window. The duration of this batching window is denoted by $\Tbatch$.  
A request denotes one or more riders scheduling a trip between two locations.
Let $Q_r$ be the number of people to be served in request $r$,
and let $\overline{Q}_v$ be the number of available seats in vehicle $v$. 
Each request $r$ has an expected travel time without ridesharing, $T_r$.
Compared with an immediate assignment approach in which a request is assigned to
a vehicle as soon as one is available, batching requests allows for matchings
that are potentially better (for both  the passengers' wait times and drivers'
travel times).

In a large-scale dynamic ridesharing problem, however, considering all available
vehicles as candidates for serving every request $r \in \cR$ can be
computationally expensive. Thus, we denote $\cV_r$ as a subset of vehicles that
can be matched with a given $r\in\cR$. While  many possible approaches
for forming $\cV_r$ exist, we let $\cV_r$ contain the $N$ vehicles closest to the
origin location of $r$ that have space available to accommodate $r$. This set can be efficiently queried from the
transportation network by using a $k$-$d$-tree-based nearest-neighbor search.

Let $T_{rv}$ be the
travel cost of assigning request $r$ to vehicle $v$.
If $v$ is idle, $T_{rv}$ is the total duration
$v$ has to travel to pick up and drop off $r$. Otherwise, $T_{rv}$ is
the additional duration $v$ has to travel to include $r$ in its route plan. The
route plan of a vehicle provides the order in which pick-ups and drop-offs are
executed by the vehicle for the requests being served by it.  Let $\Theta^0_v$
be the existing route plan of vehicle $v$, wherein $v$ could be already
catering to one or more requests, and let $\Theta^r_v$ be the potential
route plan where $v$ accommodates $r$. Then $T_{rv}$ is computed as the excess
travel time on route plan $\Theta^r_v$ compared with $\Theta^0_v$. Ideally, the
route plans are to be determined by computing constrained shortest paths, an
expensive affair that we circumvent by heuristically predetermining the pick-up
and drop-off order based on Euclidean or Manhattan distance. In other words,  starting from
the current vehicle location, the next pick-up or drop-off location is chosen
based on its closeness to the vehicle's previous location, all the while making
checks that the drop-off for a request can happen only if its corresponding pick-up
has already occurred. 

We let $D^P$ denote the maximum allowed pick-up delay for all requests.
If a request is not satisfied within $D^P$ time units of when it was made,
then the request is left to find an alternative transportation mode (e.g., bus,
rail, or an individually owned car). Let $D^T_r$ be the maximum allowed travel
delay for request
$r$; this will be set by 
\begin{equation}
    D^T_r := \max \{\overline{D}, \ T_r \times D^{\%}\},
\end{equation}
where $\overline{D}$ is a fixed minimum value for the travel delay threshold, say 10
minutes, and $D^{\%}$ is a threshold as a percentage of the expected travel
time. For example, if $D^{\%}$ is $0.33$, a maximum of $\approx 20$ minutes
travel delay is allowed for a request with an expected travel time of $1$ hour. \cref{lp} details the sets, parameters, and decision variables to be used
throughout. 

\begin{table}[!ht]
  \footnotesize
  \caption{Sets, parameters, and variables describing the ridesharing problem}
  \label[tab]{lp}
  \begin{tabularx}{\textwidth}{lX}
    \toprule
    Set & Definition\\
    \midrule
			$\cR$ & Set of requests in the batching period.\\
			$\cV_r$ & Subset of vehicles that can potentially service request $r$.\\
			$\cV$ &  Set of vehicles that can be assigned in the batching period, $\cV = \cup_{r \in \cR} \cV_r$.\\
			$\cE$ & Set of edges $(r, v)$ such that $r \in \cR, v \in \cV_r$.\\
    \midrule\\[-.35cm]
    Parameter & Definition\\
    \midrule
            $\Tbatch$ & Batching period duration.\\
            $T_r$ & Expected travel time without ridesharing for request $r$, $\forall r \in \cR$.\\
			$N$ & Number of potential vehicles to consider in $\cV^r, \forall r \in \cR$.\\
			$D^P$ & Maximum allowed pick-up delay for any request.\\
			$D^T_r$ & Maximum allowed travel delay for request $r$, $\forall r \in \cR$.\\
            $Q_r$ & Number of customers requesting a ride in request $r$, $\forall r \in \cR$.\\
            $\overline{Q}_v$ & Seats available in vehicle $v$, $\forall v \in \cV$.\\
            $\Theta^0_v$ & The route plan followed by vehicle $v$ before making assignments in the current batch, $\forall v \in \cV$.\\
            $\Theta^r_v$ & The potential route plan of vehicle $v$ if $v$ has to accommodate request $r$, $\forall (r, v) \in \cE$. \\
            $T_{rv}$ & Additional duration vehicle $v$ would have to travel 
            to serve request $r$, that is, the difference in travel time on route plan $\Theta^r_v$ and $\Theta^0_v$, $\forall (r, v) \in \cE$. \\
	\midrule\\[-.35cm]
	Variable & Definition\\
    \midrule
            $x_{rv}$ & 1 if request $r$ is assigned to vehicle $v$, $\forall (r, v) \in \cE$, 0 otherwise\\
    \bottomrule
  \end{tabularx}
\end{table}

\subsection{Assignment problem and relaxation}
\label[sec]{assign}
Matching requests to vehicles can be viewed as an assignment problem
constrained by vehicle capacities.
Two setbacks can arise when modeling the ridesharing problem in a large-scale
transportation network as an assignment problem: (i) solving a
large integer program at the end of every batching period can become a
computational bottleneck, and (ii) the cost of assigning a request to a vehicle
depends on the set of requests assigned to this vehicle in the same batching
period. This cost dependence does not appear in traditional assignment problems.
We propose the \texttt{IA} process, motivated by the iterative
approximation algorithm for the generalized assignment problem in
\citet{lau2011iterative}, to address both of these concerns. 

Assignment problems can be solved by forming a bipartite graph. In our case we let 
$G = (\cR \cup \cV, \cE)$, where $\cV = \cup_{r \in \cR} \cV_r$ is the set of vehicles
considered in the batch and $\cE = \{(r, v): r \in \cR, v \in \cV_r \}$ is the
set of edges in $G$. Binary variables
$x_{rv} \in \{0,1\}, \forall (r, v) \in \cE$ will be used to denote whether
request $r$ is assigned to vehicle $v$ in the integer program. We let LP($G$) be the linear relaxation (allowing for fractional values for
$x_{rv}$) of this integer linear assignment program on $G$, namely,

\begin{numcases}{\text{LP}(G) \coloneqq}
  \eqmathbox[lhs][l]{\minimize_{x_{rv}}} \sum_{(r, v) \in \cE} T_{rv}  x_{rv} \label{eq:LPG_obj}\\
  \eqmathbox[lhs][l]{\text{subject to}} \sum_{v : (r, v) \in \cE} x_{rv} = 1,\ \forall r \in \cR            \label{eq:LPG_cons1}\\
  \eqmathbox[lhs][l]{ } \sum_{r: (r, v) \in \cE} Q_r  x_{rv} \leq \overline{Q}_v,\ \forall v \in \cV \label{eq:LPG_cons2}\\
  \eqmathbox[lhs][l]{ } x_{rv} \geq 0,\ \forall (r, v) \in \cE, \label{eq:LPG_non_neg}
\end{numcases}


\noindent where the objective function \eqref{eq:LPG_obj} minimizes the total additional duration vehicles need to travel to serve their corresponding requests; constraints \eqref{eq:LPG_cons1}, which we will denote as $\cC_{\cR}$, impose that every request should be matched to only one vehicle (if $x_{rv}\in \{0,1\}, \forall (r, v) \in \cE$); constraints \eqref{eq:LPG_cons2}, which we will denote as $\cC_{\cV}$, ensure that vehicle capacities are not exceeded; constraints \eqref{eq:LPG_non_neg} are the non-negativity constraints.

\subsection{Iterative assignment algorithm}
\label[sec]{algo}
We can solve one or more linear programs at the end of every
batching period instead of one large-scale integer program, provided the size
of the linear programs is reduced over successive iterations. \Cref{alg:IA}
presents the \texttt{IA} algorithm that
attempts to do the same. Given an initial set of requests $\cR^0$ and vehicles $\cV^0$ resulting in the bipartite graph $G^0$, the \texttt{IA} proceeds by removing infeasible variables
(edges) in $G^0$; that is,  variable $x_{rv}$ is removed if the number of people
requesting a ride in request $r$ exceeds the number of seats available in
vehicle $v$ or if serving $r$ increases the pick-up or travel delay of any of
the customers in $v$ beyond a threshold as detailed in \cref{notation}.  After finishing the initial pruning step, if there are no feasible vehicles for a
request $r$, that is, $\deg(r) = 0$ in $G^0$, then $r$
cannot be assigned in the current batch and is removed from $\cR^0$. We will use LP($G$) to refer to the linear programs being solved within
iterations of the \texttt{IA}, as LP($G^0$) is used to refer to the original linear
program, and $G$ is a subset of $G^0$ in every iteration. For ease of
presentation, we assume in this subsection that each LP($G$) encountered is
feasible. In \cref{infeasibility}, we demonstrate what to do when an infeasible
instance of LP$(G)$ is encountered. Assuming the linear program is feasible, the \texttt{IA} solves LP($G$) and then makes use of the optimal solution thus obtained to intelligently assign a subset of requests before moving on to the next iteration to address the requests that remain. 

We will now explore how to make the assignments after solving
LP($G$) in every iteration of the \texttt{IA}. We could make a request $r$
to vehicle $v$ assignment if variable $x_{rv}$ in the optimal solution takes value $1$;
but if we end up assigning multiple requests to a single vehicle within the
same iteration, the cost of assignment will be incorrect, which was our second
concern alluded to in \cref{assign}. For example,   if
two requests $r^1$ and $r^2$ are assigned to vehicle $v$, the cost of
assignment is not equal to $T_{r^1v} + T_{r^2v}$. Hence,
instead of making all $r-v$ assignments for which $x_{rv}$ takes value $1$ in
the optimal solution, for those vehicles that potentially get multiple
requests assigned, we only assign the request with the least cost of assignment (breaking ties arbitrarily); 
the remaining requests are pushed to the next iteration. For vehicle $v$ that was assigned a request $r$, we update the route plan of $v$ to reflect the assignment, that is, $\Theta^0_v \leftarrow \Theta^r_v$. Before the \texttt{IA}
proceeds to the next iteration, the assignment costs for all the vehicles that
were assigned a request in the current iteration are recomputed. For example,   if
requests $r^1$ and $r^2$ were feasible for vehicle $v$ and if $r^1$ was
assigned to $v$ and $r^2$ remains unassigned in the current iteration, then
$T_{r^2v}$ has to be recomputed if $v$ is still feasible for $r^2$, after updating the existing route plan of $v$ to reflect the $r^1-v$ assignment, that is, $\Theta^0_v \leftarrow \Theta^{r^1}_v$. Doing so maintains the validity of the cost of all assignments within all iterations of the \texttt{IA}. The edges $(r, v) \in \cE$ that become infeasible at the end of an iteration are removed from $G$; and if there are no feasible vehicles for a request $r$, that i, $\deg(r) = 0$ in $G$, then $r$ cannot be assigned a vehicle in the current batch and is removed from $\cR$. The requests that were initially present in $\cR^0$ and are not assigned a vehicle in the final matching $\cM$ that is returned when \Cref{alg:IA} terminates can return in the next batch if their pick-up-delay threshold permits; otherwise, they are left to find alternative transportation modes.

To show that the \texttt{IA} terminates within a deterministic maximum number of
iterations, we make use of \cref{lemma_one}. If the
optimal extreme point solution of LP($G$) contains no variable
$x_{rv}$ that takes value 1 in an iteration, \cref{lemma_one} guarantees that if
LP($G$) is feasible, there exists at least one variable that takes a
value of $0$. In such an iteration of the \texttt{IA}, we remove all edges $(r, v) \in
\cE$ for which $x^*_{rv} = 0$ and then proceed to the next iteration. We have
now ensured that with every iteration the problem instance reduces by at least
one request (node) or variable (edge), thereby asserting that the \texttt{IA}
terminates in a finite number of iterations on the order of size  $G^0$.

\begin{algorithm2e}[!ht]
  \caption{Iterative Assignment (\texttt{IA}) \label{alg:IA}}
  \fontsize{9}{9}\selectfont
	\DontPrintSemicolon 
	\SetAlgoNlRelativeSize{-5}
	\SetKw{true}{true}
	\SetKw{break}{break}

    Inputs: $\cR^0, \cV^0, \cE^0$
    
     Remove infeasible edges from $\cE^0$ \tcp*[f]{Based on capacity and delay thresholds (see \cref{algo} for details)}
    
  $\cM \leftarrow \emptyset;\ \cR \leftarrow \cR^0;\ \cV \leftarrow \cV^0;\ \cE \leftarrow \cE^0$\\
  \While{$\cE \not = \emptyset$}
  {
     $\cV^a \leftarrow \emptyset$ \tcp*[f]{Set of vehicles that are assigned a request in the current iteration}

    \eIf{\textup{LP$(G)$ is infeasible}}
    {
    
        $\cR, \cV, \cE, G_U, \cM \gets $ \Cref{alg:infeasiblity}$(\cR, \cV, \cE, \cM)$

        Solve $\overline{\textrm{LP}}(G_U)$ to obtain an optimal extreme point solution $\bx^*$
    }
    {
    Solve LP$(G)$ to obtain an optimal extreme point solution $\bx^*$
    }
    
    \uIf{$\exists (r, v) \in \cE: x^*_{rv} = 1$} 
    {
        \For{$v^0 \in \cV: \exists r \in \cR \textnormal{ with } x^*_{rv^0} = 1$}
        {
                $r^0 \leftarrow \argmin_{r \in \cR} \{T_{rv^0}: x^*_{rv^0} = 1 \}$
                
                $\cM \leftarrow \cM \cup (r^{0}, v^0);\ \cV^a \leftarrow \cV^a \cup \{v^0\};\ \cR \leftarrow \cR \setminus \{r^{0}\};\ \overline{Q}_{v^0} \leftarrow \overline{Q}_{v^0} - Q_{r^{0}};\ \Theta^0_{v^0} \leftarrow \Theta^{r^0}_{v^0}$
                
                $\cE \leftarrow \cE \setminus \{(r^0, v): v \in \cV, (r^0, v) \in \cE \}$
        }
    }
    \ElseIf{$\exists (r, v) \in \cE: x^*_{rv} = 0$}
    {
        $\cE \leftarrow \cE \setminus \{(r, v): x^*_{rv} = 0 \}$
    }
    
    \For{$v \in \cV^a$}
    {
        \For{$r \in \cR: (r, v) \in \cE$}
        {
            \If{\textup{edge $(r, v)$ is infeasible}}
            {
                $\cE \leftarrow \cE \setminus \{(r, v) \}$
            }
            \Else
            {
                Recompute $T_{rv}$
            }
        }
    }
    
    $\cR \leftarrow \{r\in \cR: \deg(r) \neq 0\}$
    
    $\cV \leftarrow \{v \in \cV: \deg(v) \neq 0\}$
    
 }
     \Return $\cM$
\end{algorithm2e}

\begin{lemma}
  \label[lem]{lemma_one}
  Given $Q_{r} \leq \overline{Q}_{v},\ \forall (r, v) \in \cE$, let $\bx$ be an extreme point solution of \textup{LP($G$)}. Then
  $\exists (r, v) \in \cE: x_{rv} \in \{0, 1\}$.
\end{lemma}
\begin{proof}
Let us prove by contradiction. Assume there exists some LP($G$) with
extreme point solution $\bx$ such that $0 < x_{rv} < 1, \forall (r, v) \in
\cE$. The constraint set $\cC_{\cR}$ tells us that $\deg(r) \geq 2,\
\forall r \in \cR$ since an all-fractional solution is infeasible for LP($G$) if there exists a request $r$ with $\deg(r) = 1$. We have that
\begin{align}
    |\cE| = \frac{1}{2} \big (\sum_{r \in \cR} \deg(r) + \sum_{v \in \cV} \deg(v) \big) \geq |\cR| + \frac{1}{2} \sum_{v \in \cV} \deg(v) . \label[cons]{eq:lemma1-eq1}
\end{align}

\noindent Since $\bx$ is an extreme point solution, we know that there are at least $|\cE|$ linearly independent constraints that are active at $\bx$, that is, satisfied as equality. Let us now partition $\cV$ into two groups $\cV^{\alpha}$ and $\cV^{\beta}$ such that $\deg(v) = 1, \forall v \in \cV^{\alpha}$ and $\deg(v) \geq 2, \forall v. \in \cV^{\beta}$. Also, let $|\cV^{\alpha}| = \alpha$ and $|\cV^{\beta}| = \beta$.  Say, for some $\hat{v} \in \cV^{\alpha}$, that $\hat{r}$ is the only connected request, that is, $(\hat{r}, \hat{v}) \in \cE$. We know that the constraint $Q_{\hat{r}} x_{\hat{r}\hat{v}} \leq \overline{Q}_{\hat{v}}$ cannot be active since its given $Q_{\hat{r}} \leq \overline{Q}_{\hat{v}}$, and by assumption $0 < x_{\hat{r}\hat{v}} < 1$. Hence, $\forall v \in \cV^{\alpha}$, the corresponding constraint in $\cC_{\cV}$ cannot be active.  In addition to these constraints in $\cC_{\cV}$, the non-negativity constraints in \eqref{eq:LPG_non_neg} are also not active, which gives us the inequality
\begin{align}
    |\cR| + |\cV| - \alpha &\geq |\cE|, \label[cons]{eq:lemma1-eq2}
\end{align}

\noindent because only a maximum of $ |\cR| + |\cV| - \alpha$ constraints can be active at $\bx$. Now, from \labelcref{eq:lemma1-eq1} and \labelcref{eq:lemma1-eq2} we have
\begin{align}
    |\cR| + |\cV| - \alpha &\geq |\cR| + \frac{1}{2} \sum_{v \in \cV} \deg(v) \\
    \implies |\cV| - \alpha &\geq \frac{1}{2}\sum_{v \in \cV^{\alpha}}\deg(v) + \frac{1}{2}\sum_{v \in \cV^{\beta}}\deg(v) \\
   &= \frac{\alpha}{2} + \frac{1}{2} \sum_{v \in \cV^{\beta}} \deg(v) \\
   \implies \beta &\geq \frac{\alpha}{2} + \frac{1}{2} \sum_{v \in \cV^{\beta}} \deg(v) 
   \geq \frac{\alpha}{2} + \beta. \label[cons]{eq:lemma1-eq3}
\end{align}

\noindent From \labelcref{eq:lemma1-eq3} we see that $\alpha = 0$, which together with \labelcref{eq:lemma1-eq1} implies
\begin{align}
    |\cE| &\geq |\cR| + |\cV|.
\end{align}
\noindent In other words, all the constraints in $\cC_{\cR}$ and $\cC_{\cV}$ must be active. In order to complete the proof, it suffices to show that the constraints in $\cC_{\cR}$ and $\cC_{\cV}$ together are  linearly dependent . Let $\mathbf{A}$ be the constraint matrix of LP($G$), where $\mathbf{a}_{r}$ is the row of $\mathbf{A}$ corresponding to request $r$ in $\cC_{\cR}$ and $\mathbf{a}_{v}$ is the row corresponding to vehicle $v$ in $\cC_{\cV}$. We have that $\sum_{r \in \cR} Q_{r} \mathbf{a}_{r} = \sum_{v \in \cV} \mathbf{a}_{v}$; in other words, the maximum number of linearly independent constraints  that can be active at $\bx$ is strictly lower than $|\cR| + |\cV|$, thereby arriving at the required contradiction.
\end{proof}

\subsection{Addressing \textup{LP($G$)} infeasibility}\label[sec]{infeasibility}
It can happen that LP($G$) is infeasible in an iteration of the \texttt{IA}; that is,  given the set of requests $\cR$ and the set of vehicles $\cV$, there does not exist an assignment that assigns all the requests in $\cR$. One can easily see this happening when, for example, two requests $r^1$ and $r^2$ with $Q_{r^1} = Q_{r^2} = 1$ and $\deg(r^1) = \deg(r^2) = 1$ have $v^0$ as the only feasible vehicle for both $r^1$ and $r^2$ with $\overline{Q}_{v^0} = 1$.  

\begin{algorithm2e}[!ht]
  \caption{$G$ Reduction \label{alg:infeasiblity}}
  \fontsize{9}{9}\selectfont
	\DontPrintSemicolon 
	\SetAlgoNlRelativeSize{-5}
	\SetKw{true}{true}
	\SetKw{break}{break}

    Inputs: $\cR, \cV, \cE, \cM$\\
    
    $\cR^{\geq 2} \leftarrow \{r \in \cR: Q_r \geq 2 \}$
    
    Order $\cR^{\geq 2}$ in the decreasing order of $Q_r$
    
    \For{$r^0 \in \cR^{\geq 2}$}
    {
        \If{$\exists v \in \cV: (r^0, v) \in \cE$}
        {
            $v^0 \leftarrow \argmin_{v \in \cV} \{T_{r^0v}: (r^0, v) \in \cE \}$
            
            $\cM \leftarrow \cM \cup \{(r^0, v^0) \};\ \cR \leftarrow \cR \setminus \{r^0\};\ \overline{Q}_{v^0} \leftarrow \overline{Q}_{v^0} - Q_{r^0};\ \Theta^0_{v^0} \leftarrow \Theta^{r^0}_{v^0}$
            
            $\cE \leftarrow \cE \setminus \{(r^0, v): v \in \cV, (r^0, v) \in \cE \}$
            
            \For{$r \in \cR: (r, v^0) \in \cE$}
            {
                \If{\textup{edge $(r, v^0)$ is infeasible}}
                {
                    $\cE \leftarrow \cE \setminus \{(r, v^0) \}$
                }
                \Else
                {
                    Recompute $T_{rv^0}$
                }
            }

        }
    }
    
    $\cR \leftarrow \{r\in \cR: \deg(r) \neq 0\}$; $\cV \leftarrow \{v \in \cV: \deg(v) \neq 0\}$
    
    $G_U \leftarrow (\cR \cup \cV, \cE)$
    
    \Return  $\cR, \cV, \cE, G_U, \cM$
\end{algorithm2e}

When LP($G$) is infeasible, instead of declaring infeasibility and leaving all
the requests in $\cR$ unassigned or making an entirely greedy assignment in
the iteration, we can conduct an informed assignment by making use of the
structure of the polytope defined by the constraints of LP($G$). Consider
greedily assigning, in the decreasing order of $Q_r$, requests with $Q_r \geq
2$. For a request $r^0 \in \cR: Q_{r^0} \geq 2$, we assign vehicle $v^0$ such
that $v^0 = \argmin_{v \in \cV} \{T_{r^0v}: (r^0, v) \in \cE\}$, and we add
the assignment to the matching $\cM$ that is returned when \Cref{alg:IA}
terminates. Next we update $v^0$'s capacity and route plan, namely, $\Theta^0_{v^0} \leftarrow \Theta^{r^0}_{v^0}$, after which $\forall r \in \cR: (r,
v^0) \in \cE$, edge $(r, v^0)$ is removed from $\cE$ if it is infeasible as
per the capacity and delay thresholds; otherwise the cost of assignment
$T_{rv^0}$ is recomputed. If any request with $Q_r \geq 2$ is left unassigned
with no feasible vehicles, that request is removed from $\cR$ since
it cannot be assigned in the current batch. After this process has completed, $G$
that remains contains only those requests with $Q_r = 1$. This
makes LP($G$) the well-studied linear assignment problem on an
integral polytope: the constraint matrix is totally unimodular, and
the right-hand-side constraint vector is integral. We will refer to such a graph
$G$ as $G_{U}$, and we have that LP($G_{U}$), if
feasible, returns an integral optimal extreme point solution. \Cref{alg:infeasiblity} formally states the procedure to reduce $G$ to $G_U$. At this stage we do not know whether LP($G_{U}$) is feasible or not; and as we will see, we do not require that piece of information. Consider the following linear program
$\overline{\textup{LP}}$\textup{($G_{U}$)}  defined on
$G_{U}$:
\begin{numcases}{  \overline{\text{LP}}(G_U)\coloneqq}
  \eqmathbox[lhs][l]{\maximize_{x_{rv}}} \sum_{(r, v) \in \cE} (M - T_{rv}) x_{rv} \label{eq:LPGU_obj}\\
  \eqmathbox[lhs][l]{\text{subject to}} \sum_{v : (r, v) \in \cE} x_{rv} \leq 1,\ \forall r \in \cR            \label{eq:LPGU_cons1}\\
  \eqmathbox[lhs][l]{ } \sum_{r: (r, v) \in \cE} x_{rv} \leq \overline{Q}_v,\ \forall v \in \cV \label{eq:LPGU_cons2}\\
  \eqmathbox[lhs][l]{ } x_{rv} \geq 0,\ \forall (r, v) \in \cE, \label{eq:LPGU_non_neg}
\end{numcases}
where $M$ in \eqref{eq:LPGU_obj} is a large enough positive value such that $(M - T_{rv}) > 0$,
$\forall (r, v) \in \cE$. We have that
$\overline{\textup{LP}}$\textup{($G_{U}$)} is a maximization problem
with all positive costs, with the constraints $\cC_{\cR}$ 
in \eqref{eq:LPGU_cons1} being inequalities compared with equalities in
\eqref{eq:LPG_cons1}. Since $x_{rv} = 0$, $\forall (r, v) \in
\cE$ is a feasible solution to
$\overline{\textup{LP}}$\textup{($G_{U}$)}, we have that
$\overline{\textup{LP}}$\textup{($G_{U}$)} is always feasible. Even
when LP($G_{U}$) is infeasible,
$\overline{\textup{LP}}$\textup{($G_{U}$)} will make an informed
potential assignment as the LP solution, which is then handled as usual within the
\texttt{IA}. For cases where LP($G_{U}$) is feasible, Lemma
\ref{lemma:two} tells us that with a carefully selected value for $M$, the
optimal solution returned by
$\overline{\textup{LP}}$\textup{($G_{U}$)} is optimal for
LP($G_{U}$) as well. Hence, in an iteration of the \texttt{IA} where it
has been identified that LP($G$) is infeasible, we first reduce
$G$ to $G_{U}$ by greedily assigning requests with $Q_r \geq
2$, and then we proceed to solve
$\overline{\textup{LP}}$\textup{($G_{U}$)}. 

\begin{lemma}
  \label{lemma:two}
  If \textup{LP($G_{U}$)} is feasible, $\exists \ M < \infty$ such that
  $\overline{\bx}^*$ optimal for
  $\overline{\textup{LP}}$\textup{($G_{U}$)} is also optimal for
  \textup{LP($G_{U}$)}.
\end{lemma}
\begin{proof}
Let $\cP$ and $\overline{\cP}$ be the
polytopes defined by the constraints of LP($G_{U}$) and
$\overline{\textup{LP}}$\textup{($G_{U}$)}, respectively. Let $\bx^*$
and $\overline{\bx}^*$ be optimal extreme point solutions to
LP($G_{U}$) and $\overline{\textup{LP}}$\textup{($G_{U}$)},
respectively.  We have that $\bx^* \in \overline{\cP}$, since $\cP \subseteq
\overline{\cP}$ because constraints $\cC_{\cR}$ in
\eqref{eq:LPGU_cons1} are relaxed versions of those in \eqref{eq:LPG_cons1}.
Let $f(\bx)$ and $\overline{f}(\bx)$ be the objective function of
LP($G_{U}$) and $\overline{\textup{LP}}$\textup{($G_{U}$)},
respectively. We have
\begin{align}
    \overline{f}(\overline{\bx}^*) & \geq \overline{f}(\bx^*) \\
    \implies \sum_{(r, v) \in \cE} (M - T_{rv}) \overline{x}^*_{rv} &\geq \sum_{(r, v) \in \cE} (M - T_{rv}) x^*_{rv} \\
    \implies M \sum_{(r, v) \in \cE} (\overline{x}^*_{rv} - x^*_{rv}) &\geq \sum_{(r, v) \in \cE} T_{rv} (\overline{x}^*_{rv} - x^*_{rv} ) \\
    &= f(\overline{\bx}^*) - f(\bx^*).
\end{align}
The equality constraints in \eqref{eq:LPG_cons1} imply that $\sum_{(r, v) \in \cE} x^*_{rv} = |\cR|$, so that we get
\begin{align}
    M \Big(\sum_{(r, v) \in \cE} \overline{x}^*_{rv} - |\cR| \Big) &\geq f(\overline{\bx}^*) - f(\bx^*) \label{eq:lemma2-eq1} \\
    \implies \sum_{(r, v) \in \cE} \overline{x}^*_{rv} &\geq |\cR| - \frac{\big(f(\bx^*) - f(\overline{\bx}^*) \big)}{M}. \label{eq:lemma2-eq2}
\end{align}

\noindent The constraints $\cC_{\cR}$ in \eqref{eq:LPGU_cons1} imply that $\sum_{(r, v) \in \cE} \overline{x}^*_{rv} \leq |\cR|$, and
together with (\ref{eq:lemma2-eq2}) we have that $f(\bx^*) -
f(\overline{\bx}^*) \geq 0$. Now, if we select a value for $M$ such that it is
guaranteed that $M > f(\bx^*) - f(\overline{\bx}^*)$, then we have $\sum_{(r, v) \in \cE} \overline{x}^*_{rv} = |\cR|$ since
$\overline{\textup{LP}}$\textup{($G_{U}$)} returns an integral
optimal extreme point solution. When $\sum_{(r, v) \in \cE} \overline{x}^*_{rv} = |\cR|$, we have that $\overline{\bx}^* \in \cP$; and
equation (\ref{eq:lemma2-eq1}) implies that $ f(\overline{\bx}^*) \leq
f(\bx^*)$, which means $\overline{\bx}^*$ is optimal for LP($G_{U}$).
Now it remains to identify a suitable value for $M$ in order to complete the proof. We know that the minimum value $f(\overline{\bx}^*)$  can possibly take is $0$
and that the maximum value $f(\bx^*)$ can take is $\sum_{r \in \cR} \max
\{T_{rv} : (r, v) \in \cE \}$. Hence, if we set $M = \Big( \sum_{r \in
\cR} \max \{T_{rv} : (r, v) \in \cE \} \Big) + \epsilon$, where $\epsilon > 0$, then $\overline{\bx}^*$ that is optimal for $\overline{\textup{LP}}$\textup{($G_{U}$)} is also optimal for \textup{LP($G_{U}$)}.
\end{proof}

\subsection{Fleet operator and customer perspective}
\label[sec]{customer}
The travel time optimization in LP($G$) uses $T_{rv}$ as the cost, minimizing
which translates to reducing the overall travel times for the fleet and hence
is an optimization from the fleet operator's perspective. This might lead to
larger travel times for the customers since the optimization in essence tries
to enhance ridesharing, often at the expense of delay to the customers. A
straightforward approach to account for customer delays in the optimization is
to add a customer-centric cost component to the optimization. Similar to
$T_{rv}$, $\forall (r, v) \in \cE$ we define $W_{rv}$ to be the
maximum delay of all customers served by vehicle $v$ on the potential route plan $\Theta^r_v$. The delay for a request is calculated as
the difference between its expected travel time without ridesharing and the sum of travel time and pick-up time on the potential route plan. Now, instead of $T_{rv}$ as the cost, we could use a convex
combination of $T_{rv}$ and $W_{rv}$ as the new cost, that is, define the cost of
assigning $r$ to $v$ as 
\begin{equation}
  \widetilde{T}_{rv} := \Lambda T_{rv} + (1 - \Lambda) W_{rv},  
\end{equation}
\noindent where $\Lambda$ is a scaling parameter in the interval  $[0, 1]$.  Here, the optimization is fleet operator-centric
when $\Lambda$ is closer to $1$ and customer-centric when $\Lambda$ is closer
to $0$. 

\subsection{Greedy real-time assignment}
We will show that the \texttt{IA} algorithm efficiently solves dynamic
ridesharing problems considered. We will also consider a simpler ``greedy'' approach 
to dynamic ridesharing to compare the trade-off between the solution quality
and computational time.
This greedy procedure, derived from the heuristic in \citep{Gurumurthy2020} and summarized in \Cref{alg:greedy}, is executed
each time a request appears. First, it seeks to assign requests to
nearby vehicles while attempting to ensure all assigned requests are picked up
within the predefined threshold. When a request comes in, the algorithm finds
the nearest feasible vehicle: a vehicle that could complete the new request and
all other requests it is already serving, within the threshold pick-up times. (This
calculation of pick-up time uses Euclidean leg distances and an average
zone-based speed updated every hour to avoid computationally intensive router
calls.)
Second, \Cref{alg:greedy} seeks to assign requests to in-progress trips that are
going in the same direction. 
This is accomplished by comparing the angle between a vehicle's current position
to its destination and the vehicle's current position to the new request's
destination. If this angle is sufficiently small (in experiments we use a
threshold of 10 degrees), the vehicle is selected for matching. Once the match is
made, all ongoing pick-ups and drop-offs for a vehicle are reordered using the
heuristic nearest-neighbor search while constraining travelers to be dropped off
only after they are picked up.  

Once a trip has been assigned, a series of delay checks occur: at the end of
each pick-up or drop-off operation, all travelers' current delay is verified. If
any traveler is close to experiencing a delay greater than $D^T_r$, that vehicle
stops accepting requests and does not show up in other queries until all
passengers have been dropped off. While this is computationally beneficial, it
does not provide a clear upper bound on $D^T_r$ but instead only limits
passenger delays from being unbounded.

Previous uses of this greedy approach \citep{Gurumurthy2021, Gurumurthy2021a} have shown to improve average vehicle occupancy and fleet efficiency and help decrease system-wide vehicle miles traveled. By comparing the proposed method in this study to a benchmark implemented within the same simulation framework (POLARIS), we will show that the \texttt{IA} efficiently solves large-scale dynamic ridesharing problem instances. 

\begin{algorithm2e}[t]
  \caption{Greedy real-time assignment for a given request $r$ \label{alg:greedy}}
  \fontsize{9}{9}\selectfont
	\DontPrintSemicolon 
	\SetAlgoNlRelativeSize{-5}
	\SetKw{true}{true}
	\SetKw{break}{break}

  \While{$r$ \tn{not assigned}}
  {
    Query nearest vehicle $v$ to the request's origin $r^o$

    \uIf{$v$ \tn{is idle}}
    {
        Assign $r$ to $v$
        }
    \Else
    {
        \uIf{$v$ \tn{has seats available}}
        {
                Let $v^{pos}$ be $v$'s current location and $v^d$ be its final destination

                \uIf{\tn{The angle between directions} $v^{pos} \rightarrow v^d$
                \tn{and} $v^{pos} \rightarrow r^d$ \tn{is small}}
                {
                    \uIf{$r$ \tn{and its already assigned travelers can be picked up within pick-up threshold} $D^P$}
                    {
                        Assign $r$ to $v$
                        }
                    \Else
                    {
                        Exclude $v$ from future nearest vehicle query for $r$
                        }
                    
                }
                \Else{
                    Exclude $v$ from future nearest vehicle query for $r$
                }
        }
        \Else
        {
            Exclude $v$ from future nearest vehicle query for $r$
        }
    }
    }
\end{algorithm2e}


\FloatBarrier
\section{Numerical Experiments}\label[sec]{experiments}
In this section we carry out multiple experiments on a POLARIS model of the city of Austin, Texas, United States. We compare the IA with the greedy approach, then study the effect of fleet size and key parameters.

\subsection{Scenario setting}
We focus our study on the areas closer to the city of Austin and its downtown area where the requests are concentrated. Figure \cref{austin} shows the Austin network. The light gray area depicts the excluded zones while the green area shows the zones included in the geofence. This network contains more than 40K links and 17K nodes. The network used is based on the local Capital Area Metropolitan Planning Organization (CAMPO) and comprises a large area containing 8 counties. 

\begin{figure}
\centering
 \includegraphics[width=0.6\textwidth]{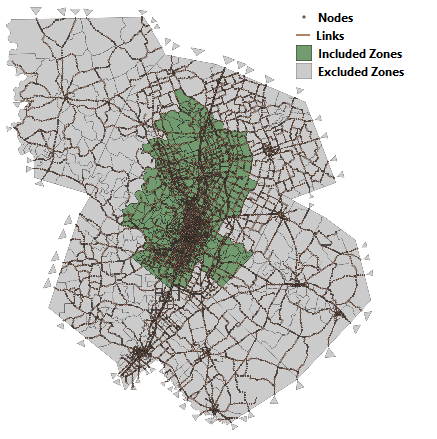}
 \caption{Austin network.}
 \label[fig]{austin}
\end{figure}

We generate the demand following the POLARIS agent-based modeling framework. The demand is an outcome of the population synthesis and its travel decisions throughout the day. The model's agents decide what activity they will do each day (work, shopping, personal, etc.) and for how long, as well as where they will do it. Based on the activity plan, agents then decide departure times and modes based on the expected travel times for each mode and their availability. Model parameters are calibrated to match the observed outcomes regarding mode shares, trip distances, departure time distributions, and the number of activities performed. We refer to the appendix in \cite{dean2022synergies} for a comparison between the simulated and observed metrics in the CAMPO area.

For this study we recorded all the TNC/taxi trips in the target area and kept the same request pattern for all the simulations. This corresponds to 153,323 requests in a 24-hour period. We also keep the same set of trips of 4.5M single-occupant vehicles  as background traffic. We run all simulations with a traffic time step of 2 seconds.

Unless otherwise stated, the batching period $B$ is 30 seconds; the number of potential to consider $N$ is 10\% of the fleet size; the objective weight $\Lambda$ is 0.5; and the maximum allowed pick-up delay for any request $D^\textrm{P}$ and the maximum allowed travel delay $D^T_r$ are 30 minutes. These are relatively high maximum allowable delays, but we keep these values because they were used for the greedy approach when calibrating the model. Moreover, the results will show that the average delays are typically much smaller than the maximum delays.

\subsection{Fleet size comparison with the greedy assignment}
In the first set of experiments, we compare the outcomes of the proposed iterative scheme with the greedy approach for different fleet sizes ranging from 1,500 to 3,000 vehicles.

\cref{fleetsize} depicts the total VHT (left), the share of requests served (middle), and the wait times (right) for the greedy (blue) and the IA (orange) for different fleet sizes. The IA yields lower VHT for all fleet sizes with a marginally higher share of requests served. The only exception is the fleet size of 3,000 vehicles in which the greedy method serves a slightly higher number of requests. Conversely, the greedy approach can keep lower wait times for all cases.

\begin{figure}[h]
\centering
 \includegraphics[width=\textwidth]{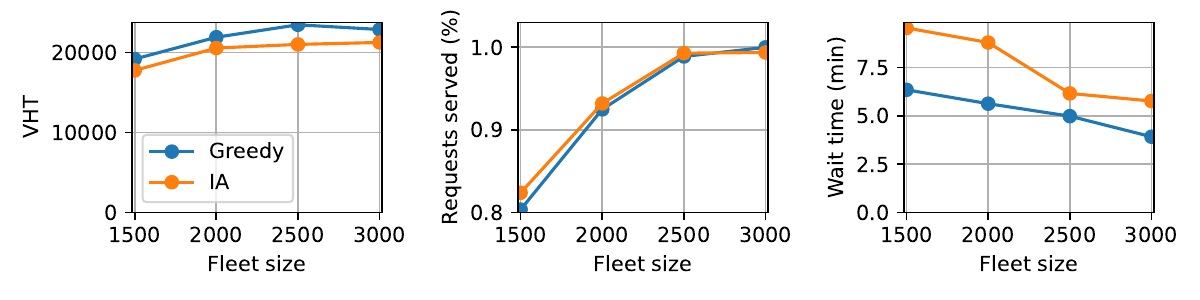}
 \caption{Comparison between the IA and the greedy approach for different fleet sizes.}
 \label[fig]{fleetsize}
\end{figure}

The complete set of metrics is reported in \cref{fleet}. The fleet size denotes the number of vehicles. The greedy approach yields lower wait and journey times, which is an outcome of its behavior of (i) assigning requests generally to the closest available vehicle, which leads to lower waiting times. and (ii) because the closest vehicle is often not a suitable match for pooling, which leads to lower journey time. However, it requires more travel per request on average and serves slightly fewer requests for lower fleet sizes. The IA processes 30-second batches in less than half a second on average. This result proves that the method is efficient and suitable for real-time use in this scenario. Note that CPU time is not reported for the greedy approach: thbatching interval concept is not needed for the greedy approach and the CPU time is negligible because the method can be is as quick as retrieving the closest vehicle from a double-ended queue (in the best case) and checking multiple vehicles (in the worst case).

\begin{table}[h]
\footnotesize
    {\begin{tabular*}\textwidth{c@{\extracolsep{\fill}}cccccc}
            \toprule
            Fleet size & Strategy & Served (\%)& Travel per req. (min) & Journey time (min) & Wait time (min) & CPU (s)\\ 
            \midrule 
            \multirow{3}{*}{1,500} 
            & \multicolumn{1}{l}{Greedy} & \multicolumn{1}{c}{80.36} & \multicolumn{1}{c}{9.36}  & \multicolumn{1}{c}{20.94} &              \multicolumn{1}{c}{6.35} & \multicolumn{1}{c}{} \\\cline{2-7}
                                 & \multicolumn{1}{l}{IA} & \multicolumn{1}{c}{82.40} & \multicolumn{1}{c}{8.44}  & \multicolumn{1}{c}{25.44} & \multicolumn{1}{c}{9.54} & \multicolumn{1}{c}{0.39} \\\cline{2-7}
                                 & \multicolumn{1}{l}{$\Delta \%$} & \multicolumn{1}{c}{+2.04} & \multicolumn{1}{c}{-9.82}  & \multicolumn{1}{c}{+21.48} & \multicolumn{1}{c}{+50.23} & \multicolumn{1}{c}{} \\\midrule
            \multirow{3}{*}{2,000} & \multicolumn{1}{l}{Greedy} & \multicolumn{1}{c}{92.49} & \multicolumn{1}{c}{9.28} & \multicolumn{1}{c}{19.56} & \multicolumn{1}{c}{5.63} & \multicolumn{1}{c}{} \\\cline{2-7}
                                 & \multicolumn{1}{l}{IA} & \multicolumn{1}{c}{93.20} & \multicolumn{1}{c}{8.64}  & \multicolumn{1}{c}{24.58} & \multicolumn{1}{c}{8.70} & \multicolumn{1}{c}{0.36} \\\cline{2-7}
                                 & \multicolumn{1}{l}{$\Delta \%$} & \multicolumn{1}{c}{+0.71} & \multicolumn{1}{c}{-6.89} & \multicolumn{1}{c}{+25.66} & \multicolumn{1}{c}{+54.52} & \multicolumn{1}{c}{} \\\midrule
            \multirow{3}{*}{2,500} & \multicolumn{1}{l}{Greedy} & \multicolumn{1}{c}{98.89} & \multicolumn{1}{c}{9.30}  & \multicolumn{1}{c}{18.06} & \multicolumn{1}{c}{4.98} & \multicolumn{1}{c}{} \\\cline{2-7}
                                 & \multicolumn{1}{l}{IA} & \multicolumn{1}{c}{99.29} & \multicolumn{1}{c}{8.29}  & \multicolumn{1}{c}{19.32} & \multicolumn{1}{c}{6.16} & \multicolumn{1}{c}{0.35} \\\cline{2-7}
                                 & \multicolumn{1}{l}{$\Delta \%$} & \multicolumn{1}{c}{+0.40} & \multicolumn{1}{c}{-10.86}  & \multicolumn{1}{c}{+6.97} & \multicolumn{1}{c}{23.69} & \multicolumn{1}{c}{} \\\midrule
            \multirow{3}{*}{3,000} & \multicolumn{1}{l}{Greedy} & \multicolumn{1}{c}{99.99} & \multicolumn{1}{c}{8.96}  & \multicolumn{1}{c}{15.81} & \multicolumn{1}{c}{3.92} & \multicolumn{1}{c}{} \\\cline{2-7}
                                 & \multicolumn{1}{l}{IA} & \multicolumn{1}{c}{99.37} & \multicolumn{1}{c}{8.38}  & \multicolumn{1}{c}{18.64} & \multicolumn{1}{c}{5.76} & \multicolumn{1}{c}{0.36} \\\cline{2-7}
                                 & \multicolumn{1}{l}{$\Delta \%$} & \multicolumn{1}{c}{-0.62} & \multicolumn{1}{c}{-6.47}  & \multicolumn{1}{c}{+17.90} & \multicolumn{1}{c}{+46.93} & \multicolumn{1}{c}{} \\
                                 \bottomrule
        \end{tabular*}}
    \caption{Key metrics for different fleet sizes using the IA and the greedy approach.}
    \label[tab]{fleet}
\end{table}

The reason for the higher wait time on the IA is that it maintains a higher vehicle occupancy than does the greedy approach. By increasing pooling with requests with fewer detours, the IA can reduce the total VHT at the expense of higher wait times. This pattern is clearer in assessing the time-series comparison. \cref{temporalgreedyia} depicts average occupancy (left), time spent on the network (middle), and average wait time (right) over the 24-hour period. Outside the peak, the proposed IA strategy consistently keeps higher vehicle occupancy. This results in slightly slower time spent in the network, higher wait time, and lower fleet size to serve the same demand.

The key performance distinction occurs during the peak period. One can observe in the left graph in \cref{temporalgreedyia} that the occupancy for both strategies spikes at around 7AM. The greedy strategy keeps the occupancy around 1.5 until noon, whereas the IA exceeds two people per vehicle. This results in an overall fewer vehicles used and therefore less time spent in the network over that period, which is clear in the middle graph. The IA strategy serves the same amount of requests while yielding approximately 8\% less VHT (see \cref{fleetsize}).

\begin{figure}[h]
\centering
 \includegraphics[width=\textwidth]{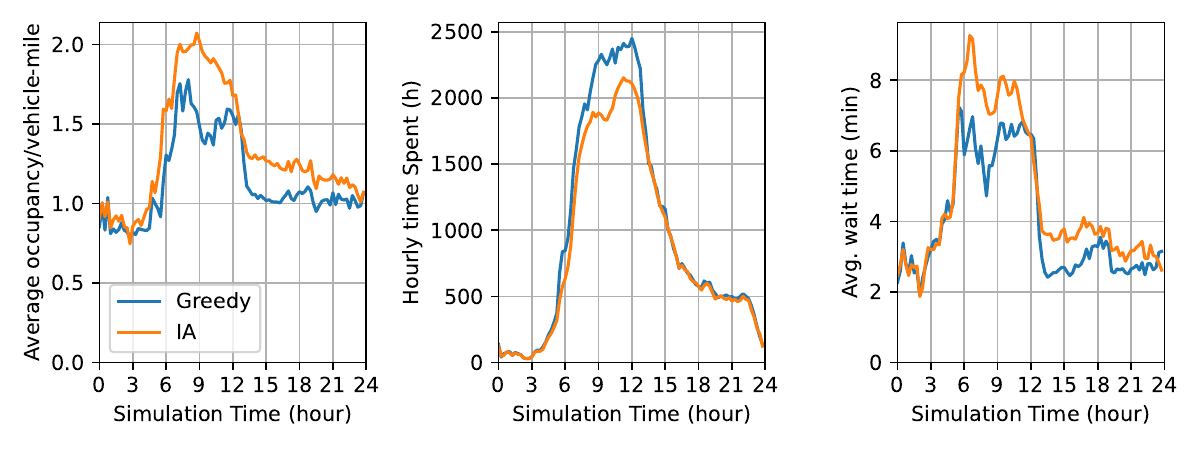}
 \caption{Temporal comparison between the IA and the greedy approach for 2,500 -vehicle case.}
 \label[fig]{temporalgreedyia}
\end{figure}

The IA strategy is aware of the need to increase pooling when there are fewer vehicles, as evidenced by the vehicle occupancy for different fleet sizes. \cref{temporalgreedyia} depicts the same metrics (occupancy, time spent, and wait time from left to right) for the IA strategy for four different fleet sizes. The average occupancy decreases following the peak earlier, the higher the fleet size is. The reduction in occupancy results in lower wait time and higher time spent during that period.

\subsection{Effect of the number of vehicles considered per request}
The IA strategy takes $N$, the number of potential vehicles to consider for each request, as an input parameter. Reducing $N$ saves computational time as it potentially decreases the number of decision variables in the optimization problem and the number of routes to be computed. On the other hand,  a higher number of vehicles per request gives more flexibility to the solver to find better solutions. In this section we assess the trade-offs for $N$ by assessing the results for $N \in \{125, 250, 500\}$ for a fleet of 2,500 vehicles. 

\cref{metricsnstudy} shows the VHT (left), the share of requests served (middle), and the average wait time (right) for the three different fleet sizes. For all three metrics, the performance changes are marginal. Increasing $N$ leads to more requests served and lower wait time, although the total VHT is slightly higher for bigger $N$. In terms of VHT per request, the results are also favorable for the $N=500$ case since it served more requests. 

\begin{figure}[h]
\centering
 \includegraphics[width=\textwidth]{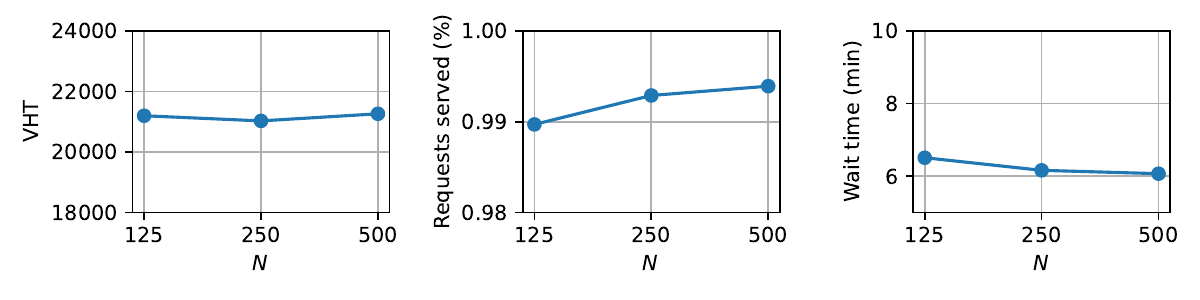}
 \caption{VHT (left), share of requests served (middle), and average wait time (right) as a function of vehicles considered per request ($N$).}
 \label[fig]{metricsnstudy}
\end{figure}

\cref{nperrequest} provides further insights for this analysis. Even though travel per request goes up with $N=500$, this case has the shortest journey and wait time compared with other cases. As expected, the CPU time increases for higher $N$. Nevertheless, the average CPU time is still low enough for real-time operation for batches of 30 seconds for this scenario.

\begin{table}[h]
\footnotesize
        {\begin{tabular*}\textwidth{c@{\extracolsep{\fill}}cccccc}
            \toprule
            Fleet size & $N$ & Served (\%) & Travel per req. (min) & Journey time (min) & Wait time (min) & CPU time (s)\\ 
            \midrule 

            \multirow{3}{*}{2,500} & \multicolumn{1}{l}{125} & \multicolumn{1}{c}{98.97} & \multicolumn{1}{c}{8.39}  & \multicolumn{1}{c}{20.64} & \multicolumn{1}{c}{6.50} & \multicolumn{1}{c}{0.16} \\\cline{2-7}
                                 & \multicolumn{1}{l}{250} & \multicolumn{1}{c}{99.29} & \multicolumn{1}{c}{8.29}  & \multicolumn{1}{c}{19.32} & \multicolumn{1}{c}{6.16} & \multicolumn{1}{c}{0.35} \\\cline{2-7}
                                 & \multicolumn{1}{l}{500} & \multicolumn{1}{c}{99.39} & \multicolumn{1}{c}{8.38}  & \multicolumn{1}{c}{18.81} & \multicolumn{1}{c}{6.07} & \multicolumn{1}{c}{0.54} \\
                                 \bottomrule

        \end{tabular*}}
    \caption{Key metrics for vehicles per request using the IA strategy.}
    \label[tab]{nperrequest}
\end{table}

\subsection{Effect of objective function weight}
We perform a similar comparison for the objective function parameter $\Lambda$ that weights between service- ($\Lambda \to 1$) and customer ($\Lambda \to 0$)-centric operation. To that end, we perform experiments for $\Lambda \in \{0.25, 0.5, 0.75\}$.

For this parameter sensitivity, the results are salient. \cref{metricslambdastudy} shows the VHT (left), the share of requests served (middle), and the average wait time (right) for the three values of $\Lambda$. Increasing $\Lambda$ leads to a significant decrease in the total VHT. On the other hand, there is a decrease in the number of requests served, and the wait time for the case of $\Lambda=0.75$ is around 10\% higher than the other two values of $\Lambda$

\begin{figure}[h]
\centering
 \includegraphics[width=\textwidth]{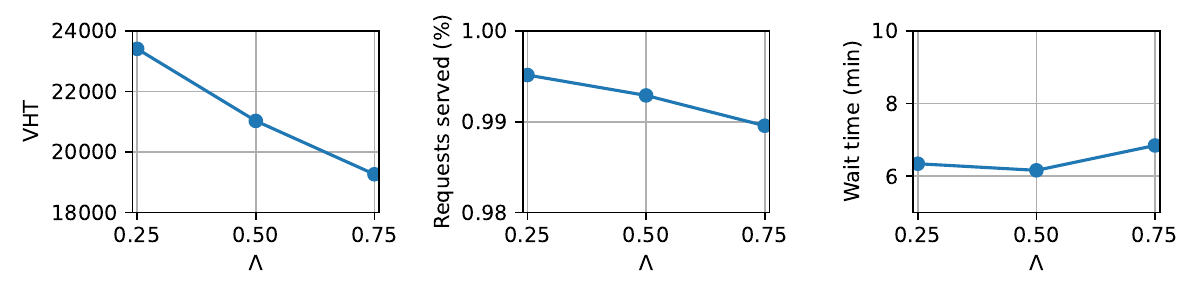}
 \caption{VHT (left), share of requests served (middle), and average wait time (right) as function of objective function weight $\Lambda$.}
 \label[fig]{metricslambdastudy}
\end{figure}

The reason behind the changes in the VHT for higher values of $\Lambda$ is the higher amount of pooling achieved in these cases. This is clear in the temporal evolution of vehicle occupancy (left), time spent in the network (middle), and average wait time (right) in \cref{temporallambda}. The average occupancy is substantially higher values of $\Lambda$, which in turn leads to lower VHT. Nevertheless, the average occupancy for $\Lambda=0.5$ is smaller than the the average occupancy for $\Lambda=0.75$, but that does not lead to significant changes in VHT during the peak period. On the other hand, the $\Lambda=0.25$ case leads to a significantly higher VHT at that period.

\begin{figure}[h]
\centering
 \includegraphics[width=\textwidth]{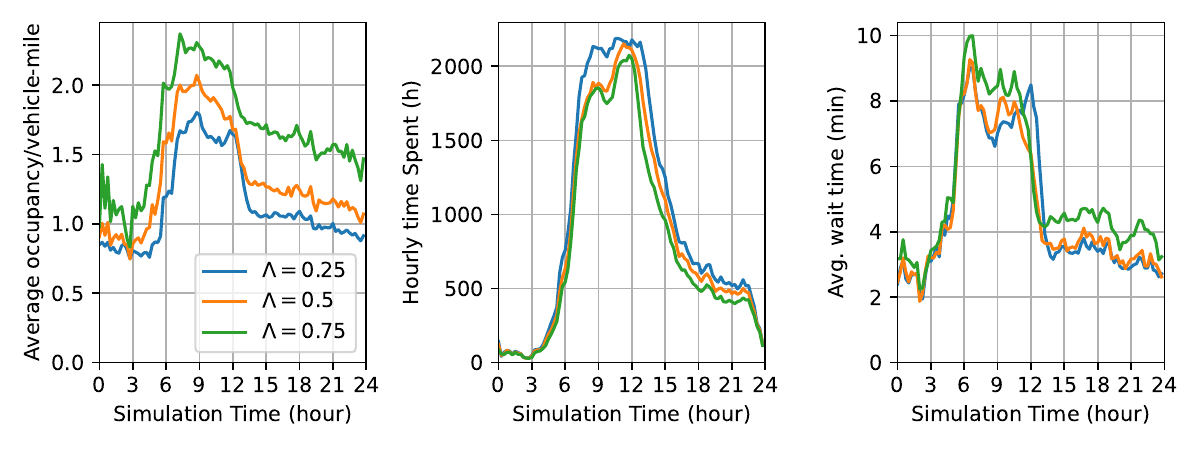}
 \caption{Average occupancy (left), hourly time spent (middle), and average wait time (right) for $\Lambda \in \{0.25, 0.5, 0.75 \}$}
 \label[fig]{temporallambda}
\end{figure}

 \cref{objectiveweight} presents the share of served requests, travel time per request, and journey, wait, and CPU times for the three different values of $\Lambda$.

\begin{table}[h]
\footnotesize
        {\begin{tabular*}\textwidth{c@{\extracolsep{\fill}}cccccc}
            \toprule
            Fleet size & $\Lambda$ & Served (\%)& Travel per req. (min) & Journey time (min) & Wait time (min) & CPU time (s)\\ 
            \hline 

            \multirow{3}{*}{2,500} & \multicolumn{1}{l}{0.25} & \multicolumn{1}{c}{99.51} & \multicolumn{1}{c}{9.21}  & \multicolumn{1}{c}{19.03} & \multicolumn{1}{c}{6.34} & \multicolumn{1}{c}{0.36} \\\cline{2-7}
                                 & \multicolumn{1}{l}{0.50} & \multicolumn{1}{c}{99.29} & \multicolumn{1}{c}{8.29}  & \multicolumn{1}{c}{19.32} & \multicolumn{1}{c}{6.16} & \multicolumn{1}{c}{0.35} \\\cline{2-7}
                                 & \multicolumn{1}{l}{0.75} & \multicolumn{1}{c}{98.96} & \multicolumn{1}{c}{7.62}  & \multicolumn{1}{c}{21.62} & \multicolumn{1}{c}{6.84} & \multicolumn{1}{c}{0.32} \\\cline{1-7}

        \end{tabular*}}
    \caption{Key metrics for vehicles per request using the IA strategy for different $\Lambda$ values.}
    \label[tab]{objectiveweight}
\end{table}

\subsection{Effect of the batching interval}
The batching interval $B$ has a potential impact on the solution quality. It does not influence the modeling; however, the set of requests for each batch is dependent upon $B$. A shorter batching interval means that each request is processed earlier, and it can lead to lower wait times for the same assignment set. On the other hand, the solver is provided less flexibility for making decisions and therefore may make worse decisions. In this study we test $B \in {15,30,60}$ seconds.

\cref{metricsbatchingstudy} shows the VHT (left), share of requests served (right), and average waiting time (right). Overall, the batching interval is the lever with the least variations in outcome. Nevertheless, there are interesting aspects to it. Increasing the batching interval leads to a VHT decrease, which was expected. However, the change in wait time is marginal and much lower than the average time for processing the batch (half of the batching interval). Therefore, the highest batching intervals would be advised in this particular case.

\begin{figure}[h]
\centering
 \includegraphics[width=\textwidth]{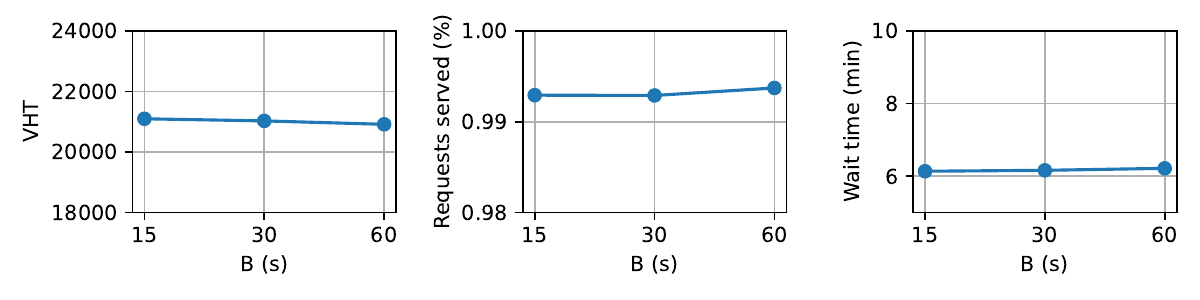}
 \caption{VHT (left), share of requests served (middle), and average wait time (right) as function of batching interval (B).}
 \label[fig]{metricsbatchingstudy}
\end{figure}

\cref{batchinterval} shows the share of served vehicles, travel per request, journey time, wait time, and CPU time for different batching intervals. The table also supports that $B$ does not have a considerable impact on the key performance metrics, yet it plays a critical role in the CPU time.

\begin{table}[h]
\footnotesize
        {\begin{tabular*}\textwidth{c@{\extracolsep{\fill}}cccccc}
            \toprule
            Fleet size & $B$ & Served (\%) & Travel per req. (min) & Journey time (min)& Wait time (min)& CPU time (s)\\ 
            \midrule 

            \multirow{3}{*}{2,500} & \multicolumn{1}{l}{15} & \multicolumn{1}{c}{99.29} & \multicolumn{1}{c}{8.32}  & \multicolumn{1}{c}{19.39} & \multicolumn{1}{c}{6.16} & \multicolumn{1}{c}{0.16} \\\cline{2-7}
                                 & \multicolumn{1}{l}{30} & \multicolumn{1}{c}{99.29} & \multicolumn{1}{c}{8.29}  & \multicolumn{1}{c}{19.32} & \multicolumn{1}{c}{6.16} & \multicolumn{1}{c}{0.35} \\\cline{2-7}
                                 & \multicolumn{1}{l}{60} & \multicolumn{1}{c}{99.37} & \multicolumn{1}{c}{8.24}  & \multicolumn{1}{c}{19.11} & \multicolumn{1}{c}{6.21} & \multicolumn{1}{c}{0.80} \\
                                 \bottomrule

        \end{tabular*}}
    \caption{Key metrics for vehicles per request using the IA strategy for different batching interval.}
    \label[tab]{batchinterval}
\end{table}

\section{Conclusion}\label[sec]{conclusion}
The DRS is a challenging problem due to the impact of service order complexity on the system performance. Moreover, the two parties involved in the problem---riders and operators---pursue different objectives. In this study we developed an optimization model considering a balanced objective to solve the DRS problem. In this problem, time-to-solution is one of the key performance metrics because vehicle assignments need to be handled quickly to satisfy riders. To this end, we proposed the IA approach for the DRS problem. The problem is cast in a series of linear programs to assign riders to vehicles. We prove the optimality of the approach even in the case of multiple party sizes. Moreover, the method is computationally efficient since it relies on linear programming. As a benchmark, we develop a greedy heuristic solution and compare the computational performance and key solution quality metrics. The IA implementation compared with the greedy method in Austin, TX, showed shorter journey time and wait time with the expense of CPU time increase.

In the studies around key parameters, we found the objective function weight is the most significant in changing the outcomes. A service-centric objective leads to lower VHT at the expense of higher average wait times. The batching interval and the number of vehicles per request have some impacts, but they are limited. The key mechanism for the different outcomes in the objective function is the average occupancy. A higher weight $\Lambda$ on the total travel time induces more pooling, which reduces VHT and increases wait and journey times.

In this study we did not consider repositioning of the vehicles. In a DRS network, the distribution of vehicles can impact the overall service quality. One may target repositioning a vehicle from its low-demand last service destination to a high-demand destination to increase the probability of a better match. While this unoccupied travel incurs cost and time, it may increase customer satisfaction with faster service. The DRS problem could incorporate repositioning decisions while making assignments.

Fossil fuel vehicles leave a large footprint on the environment, and governments now have global goals toward net-zero emissions. Vehicle electrification is a top priority among actions toward achieving these goals. In this study we did not schedule charging activities, which is a challenging task and adds to the complexity of the DRS problem. Future studies should consider these aspects to address the electrification in this industry.

\section*{Acknowledgments}
This material is based upon work supported by the U.S. Department of Energy, Office of Science, under contract number DE-AC02-06CH11357. 
This report and the work described were sponsored by the U.S. Department of Energy (DOE) Vehicle Technologies Office (VTO) under the Systems and Modeling for Accelerated Research in Transportation (SMART) Mobility Laboratory Consortium, an initiative of the Energy Efficient Mobility Systems (EEMS) Program. 
Erin Boyd, a DOE Office of Energy Efficiency and Renewable Energy (EERE) manager, played an important role in establishing the project concept, advancing implementation, and providing guidance. The authors remain responsible for all findings and opinions presented in the paper. The findings are not suggestions for agencies to implement given the assumptions made in this study.